\documentclass{amsart}
\usepackage{amssymb,amsfonts}
\usepackage{latexsym}
\usepackage{amsmath,mathrsfs}
\usepackage[all,arc]{xy}
\usepackage{enumerate}
\usepackage[english]{babel}
\usepackage[bookmarks=true]{hyperref}
\usepackage{tikz}
\usepackage{blkarray}
\newtheorem{theo}{Theorem}[section]
\newtheorem{coro}[theo]{Corollary}
\newtheorem{prop}[theo]{Proposition}
\newtheorem{lemm}[theo]{Lemma}

\theoremstyle{definition}

\theoremstyle{remark}
\newtheorem{rema}[theo]{Remark}

\newcommand{\bb}[1]{\mathbb{#1}}
\newcommand{\al}[1]{\mathcal{#1}}
\newcommand{\sr}[1]{\mathscr{#1}}
\newcommand{\ak}[1]{\mathfrak{#1}}

\newcommand{\rk}{{\rm rk}}

\newcommand{\ra}{\rightarrow}
\newcommand{\lra}{\longrightarrow}
\newcommand{\x}[1]{\text{#1}}

\begin{document}
\title{Strange duality at level one for alternating vector bundles}
\author{Hacen ZELACI}
\address{Mathematical departement. University of EL Oued.}
\curraddr{}
\email{z.hacen@gmail.com}
%\date{\today}

\begin{abstract}
In this paper, we show a strange duality isomorphism at level one for the space of generalized theta functions on the moduli spaces of alternating anti-invariant vector bundles in the ramified case. These anti-invariant vector bundles constitute one of the non-trivial examples of  parahoric $\al G-$torsors, where $\al G$ is a twisted (not generically split) parahoric group scheme.
\end{abstract}
\maketitle
\tableofcontents

\section{Introduction}
In their seminal paper (\cite{BNR}), Beauville, Narasimhan, and Ramanan established a duality isomorphism between the space of level one generalized theta functions on the moduli space of rank $r$ semistable vector bundles with trivial determinant over a smooth curve and the dual of the space of global sections of the $r$-th multiple of the Riemannian theta divisor on the Jacobian variety of the curve.  More precisly, let $\al{SU}_X(r)$ be the moduli space of rank $r$ vector bundles of trivial determinant over $X$, and denote by $\sr D$ the determinant of cohomology line bundle on $\al{SU}_X(r)$. Let $J_X$ the Jacobian variety on $X$ and $\Theta$ be a Riemann theta divisor on $J_X$. Then BNR showed that there exists a canonical isomorphism $$H^0(\al{SU}_X(r),\sr D)\cong H^0(J_X,\al O(r\Theta))^*.$$ This duality is known as the BNR duality, and it was the first case of the so-called strange duality isomorphism.

Let  $\pi:X\ra Y$ be a ramified double cover of smooth curves  and let $r\geqslant 1$ be an integer. Denote by $\sigma$ the corresponding involution on $X$ and by $R$ the ramification divisor. 
A vector bundle $E$ on $X$ is called anti-invariant if there exists an isomorphism $$\psi:\sigma^* E\ra E^*.$$ If $\sigma^*\psi=-\psi^t$, we say that $E$ is   $\sigma-$alternating. \\

The moduli space of anti-invariant vector bundles stands out as a noteworthy instance of the moduli space of parahoric $\al G$-torsors, where $\al G$ represents a parahoric group scheme over $X$, and notably, it is not generically split in this case. This underscores the significance of the study of these moduli spaces. Over the past decade, parahoric torsors have become a central focus of considerable interest among researchers, since they can be viewed as a natural generalization encompassing various other concepts, including $G$-bundles, parabolic bundles, and invariant $G$-bundles. 

The moduli spaces of anti-invariant vector bundles can also be seen as a generalization of Prym varieties to higher rank. Indeed, the action of $\sigma$ lifts to an action on the moduli space of vector bundles, and as in the case of Jacobian, one can look at the invariant and anti-invariant parts, the later one is precisely the moduli space of anti-invariant vector bundles.\\
%for each line bundle $L$ on $X$ with norm $K_Y\otimes \Delta$, where $K_Y$ is  the canonical bundle on $Y$ and $\Delta=(\pi_*\al O_X)^{-1}$, 
%the universal family over $\al{SU}_X^{\sigma,-}(r)$ twisted with $L$ is equipped with  a non-degenerated quadratic form with values in the canonical line bundle $K_Y$. Hence its 

In this paper, we prove a BNR duality for the moduli space of $\sigma-$alternating vector bundles in the ramified case.\\

The moduli spaces of semistable pairs $(E,\psi)$ of $\sigma-$alternating vector bundles of rank  $2r$, which is denoted  $\al{U}_X^{\sigma,-}(2r)$,  has been constructed in \cite{Z2}, it has two connected components $\al{U}_{X,0}^{\sigma,-}(2r)$ and $\al{U}_{X,1}^{\sigma,-}(2r)$. The case of trivial determinant is denoted $\al{SU}_X^{\sigma,-}(2r)$. To each pair $(E,\psi)$ with $\det E=\al O_X$ and $\det \psi=1$, we associate a topological invariant  $\tau$ (see Section\ref{2}).
We denote the associated moduli space by $\al{SU}_{X,\tau}^{\sigma,-}(2r)$.\\
There exists a Prym variety $\rm \tilde P$ over  the normalisation $\hat{X}_s$ of the  spectral curve  $\tilde X_s\ra X$ (more details will be given in section \ref{2}) such that the pushforward rational map $${\rm \tilde P}\ra \al U_{X,0}^{\sigma,-}(2r)$$ is dominant. Moreover, the Prym variety $\rm \tilde P$ is principally polarized. Let $\tilde{\sr L}$ be a line bundle associated to a principal polarization on $\rm \tilde P$. We can show that the codimension of the locus where this map is not defined is at least $2$. Hence we get an injection $H^0(\al U_{X,0}^{\sigma,-}(2r),\sr D)\hookrightarrow H^0({\rm\tilde P},\tilde{\sr L})$, where $\sr D$ is the determinant of cohomology line bundle on $\al U_X^{\sigma,-}(2r)$. 
 \begin{theo}
     The map $(q_*)^*:H^0(\al U_{X,0}^{\sigma,-}(2r),\sr D) \ra H^0({\rm \tilde P}, \tilde{\sr L})$ is an isomorphism. In particular, $$\dim H^0(\al U_{X,0}^{\sigma,-}(2r),\sr D)=1.$$
 \end{theo}

 Consider the moduli space $\al{SU}_{X,\tau}^{\sigma,-}(2r)$ of trivial determianant $\sigma-$alternating vector bundles. In \cite{Z3}, it has been  shown that the determinant line bundle over  $\al{SU}_{X,\tau}^{\sigma,-}(2r)$ admits a square root, called a Pfaffian of cohomology bundle $\sr P_L$. Moreover, the Hitchin system was studied on these moduli spaces and it was also shown that this line bundle satisfies a BNR duality in the unramified case.\\ % i.e. we show that there exists a canonical isomorphism between the space $H^0(\al{SU}_X^{\sigma,-}(2r),\sr P_L)$ of generalized theta functions and the dual space $H^0({\rm P},\sr L^r)^*$. \\

 Let $\rm P$ be the Prym variety of $X\ra Y$, and let $\sr L$ be a line bundle on $\rm P$ associated to a principal polarization. 
 Our main result is the following %obtained from the restriction of the Riemann theta devisor $\Theta_\kappa\subset \mathrm{Pic}^0(X)$ to $\rm P$, where $\kappa$ is theta caracteristic on $X$. 
\begin{theo}[Theorem \ref{main}] There exists a canonical isomorphism
 $$ H^0(\al{SU}_{X,\tau}^{\sigma,-}(2r), \sr{P}_L)\cong H^0({\rm P},\sr L^r)^*.$$ 
\end{theo}
Note that the power of the polarization $\sr L$ is the half of the rank of vector bundles which is  $2r$, and this is because the determinant of cohomology has a square root however, the polarization $\sr L$ doesn't have a square root since the Prym variety $\rm P$ is not principally polarized.\\
As corollary, we deduce the dimensions of these spaces. 
\begin{coro}
    $$\dim H^0(\al{SU}_{X,\tau}^{\sigma,-}(2r), \sr{P}_L)=2^{g_Y}r^{g_Y+n-1}.$$\\
\end{coro}

%\section{Preliminaries} \label{pre}
%Let  $\al{U}_X(r,0)$ be the moduli space of stable vector bundles of rank $r$. The involution $\sigma$ induces  by pullback an involution on $\al{U}_X(r,0)$. A vector bundle $E$ is called $\sigma-$symmetric (resp. $\sigma-$alternating) if there exists an isomorphism $$\psi:\sigma^*E\stackrel{\sim}{\lra } E^*.$$  such that $$\sigma^*\psi=\;^t\psi\;\,(\x{ resp. } \sigma^*\psi=-\;^t\psi).$$ Let  $\al{U}_X^{\sigma,\pm}(r)$ denote the locus of $\sigma-$symmetric and $\sigma-$alternating stable vector bundles. Note that $\al{U}_X^{\sigma,-}(2k+1)=\emptyset$, since $\psi$ induces a non degenerate alternating form on $E_p$ for a ramification point $p$.

\textbf{Acknowledgement:} I am very much indebted to C. Pauly and S. Mukhopadhyay for useful discussions about these questions. I would also like to thank Z. Ouaras. and A. Peón-Nieto.

\section{$\sigma-$alternating vector bundles}\label{2}
For the convenience of the reader, we recall here some properties of $\sigma-$anti-invariant vector bundles.\\
Let $X$ be a complex smooth projective curve of genus $g_X$ with an involution $\sigma:X\ra X$ such that ${\rm Fix}(\sigma)\not=\emptyset$. Denote by $Y=X/\sigma$ and let $g_Y$ be its genus. We assume that $g_Y\geqslant2$. Let $K_X$ and $K_Y$ be their  canonical bundles.  The Hurwitz formula gives $$g_X=2g_Y+n-1,$$ where $n$ is the half of the number of ramification points in $X$. 

A vector bundle $E$ over $X$ is called $\sigma-$anti-invariant if there exists an isomorphism $\psi:\sigma^*E\ra E^*$. If $E$ is stable, then this  last isomorphism is either symmetric  $\sigma^*\psi=\psi$ or alternating $\sigma^*\psi=-\psi$. In the first case, we say that $E$ is $\sigma-$symmetric, and in the other case, namely $\sigma^*\psi=-\psi$, then $E$ is called $\sigma-$alternating. Note that, since the cover is supposed to be ramified, there is no $\sigma-$alternating bundle with odd rank and this is because $\psi$ induces over a ramification point an anti-symmetric linear isomorphism $$\psi_p:E_p\ra E_p^*,$$ which implies that the dimension of $E_p$ is even. In this paper, we mainly deal with $\sigma-$alternating vector bundles.

A pair $(E,\psi)$ is called semistable if for any $\sigma-$isotropic subbundle $F\subset E$, we have $$\dfrac{\deg F}{\rk F}\leqslant \dfrac{\deg E}{\rk E},$$ where a $\sigma-$isotropic subbundle is a vector subbundle $F$ such that the following composition  map $$\sigma^*F\hookrightarrow \sigma^*E\ra E^*\twoheadrightarrow F^*$$ is zero. The  moduli space of semistable pairs $(E,\psi)$ of $\sigma-$alternating vector bundles and a $\sigma-$alternating isomorphism, which we denote $\al{U}_X^{\sigma,-}(2r)$, has been constructed in a previous work of the author (\cite{Z2}). \\

 It was shown in \cite{Z} that the moduli space  $\al{U}_{X}^{\sigma,-}(2r)$ of  $\sigma-$alternating bundles has two connected components distinguished by the parity of $h^0(E\otimes \kappa)$ where $\kappa$ is an even theta characteristic over $X$.  Note the similarity with  the Stiefel-Whitney class that distinguishes the connected components of the moduli space of orthogonal bundles. \\ We denote by $\al{U}_{X,0}^{\sigma,-}(2r)$ and $\al{U}_{X,1}^{\sigma,-}(2r)$ these two connected components with the obvious meaning of notations. %The stacks are denoted by $\sr{U}_x^{\sigma,-}(r)$.
 
 Consider now the moduli space $\al{SU}_{X,s}^{\sigma,-}(2r)$ of \emph{stable} $\sigma-$alternating vector bundles with trivial determinant. This moduli space   has $2^{2n-1}$ connected components, half of them are contained in $\al{U}_{X,0}^{\sigma,-}(2r)$ and they are distinguished by a topological type defined as follows: let $(E,\psi)$ be a stable $\sigma-$alternating vector bundle such that $\det E=\al O_X$ and $\det(\psi)= 1$. Then over a ramification point $p\in R$, $\psi$ induces an anti-symmetric isomorphism $\psi_p:E_p\ra E_p^*$ with determinant equals $1$, hence its Pfaffian $\rm Pf(\psi_p)= \pm 1$. We define the type of $(E,\psi)$ to be $$\tau=({\rm Pf}(\psi_p))_{p\in R}\mod \pm1. $$  Note that this is well defined, i.e. it doesn't depend on the choice of $\psi$ as long as $E$ is supposed to be stable. We denote $\al{SU}_{X,\tau}^{\sigma,-}(2r)$ the  moduli space of semistable $\sigma-$alternating vector bundles $(E,\psi)$ of type $\tau$.
%Here we study  of connected components of the case of trivial determinant $\sigma-$alternating vector bundles.

Denote by $\Delta=(\det\pi_*\al O_X)^{-1}.$ Note that $\Delta$ is a degree $n$ line bundle on $Y$.
\begin{lemm}\label{descendpfaffian}
Let $L$ be a line bundle on $X$ with norm $K_Y\otimes \Delta$, then 
the determinant line bundle $\sr D$ over  $\al{U}_X^{\sigma,-}(2r)$ admits a square root, called a Pfaffian of cohomology line bundle and denoted $\sr P_L$. 
	%Assume that $\pi$ is ramified. Let $\kappa$ be an even $\sigma-$invariant theta characteristic on $X$, then 
	%The Pfaffian line bundle $\sr P_L$ over $\sr{U}_{X,0}^{\sigma,-}(2r)$ descends to $\al{U}_{X,0}^{\sigma,-}(2r)$.
\end{lemm}
\begin{proof}
    Let $\sr {U}_X^{\sigma,-}(2r)$ be the stack of semi-stable $\sigma-$alternating vector bundles and let $\sr U$ be a universal family over $\sr{U}_X^{\sigma,-}(2r)\times X$. Let $\sr U_L:=\sr U\otimes q_2^*L$, where $q_2:\sr{U}_X^{\sigma,-}(2r)\times X\ra X$ is the second projection. Then the family $\pi_*\sr{U}_L$ has a quadratic form with values in $K_Y$. Indeed, we have 
    \begin{align*}
        \pi_*\sr{U}_L& \cong \pi_*(\sigma^*\sr{U}_L)\\ & \cong \pi_*(\sr{U}_L^*\otimes q_2^{-1}(R))\otimes q_2^*K_Y \\ &\cong \pi_*(\sr{U}_L)^*\otimes q_2^*K_Y,
    \end{align*}
where the last isomorphism is giving by the relative duality (see \cite[ Ex III.6.10]{HA}). This isomorphism induces a symmetric bilinear form on the family with values in $K_Y$ (see \cite{Z3}).  Using \cite[Proposition $7.9$]{So}, we obtain a square root of the determinant of cohomology $\sr D_{\pi_*\sr U_L}$ associated $\pi_*{\sr U_L}$. But since $\sr D_{\pi_*\sr U_L}=\sr D_{\sr U_L}$, we deduce that the determinant line bundle has square root at the level of stacks.\\ To show that this square root descends to the moduli space, we need to use the Kempf's lemma at the level of Quot scheme. This can be done as in \cite{Z3}.
\end{proof}
The Pfaffian line bundle admits a nonzero global section, which insures the existence of a Pfaffian divisor. For  $L\in \x{Nm}_{X/Y}^{-1}(K_Y\otimes\Delta)$, denote by $\Theta_L$ the divisor in $\al{U}_X(2r,0)$ supported on vector bundles $E$ such that $E\otimes L$ has a non-zero global section. 

\begin{lemm}\label{effective}
	For any $L$ as above, the  restriction of the divisor $\Theta_{L}\subset \al U_X(2r,0)$ to $\al{U}_{X,0}^{\sigma,-}(r)$  is again a divisor.
	 Moreover there exists an effective divisor  $\Xi_L$ in $\al{U}_X^{\sigma,-}(2r)$ such that  $\al O(\Xi_L)\cong\sr P_L$ and  $$2\Xi_L\equiv\Theta_L.$$
\end{lemm}
\begin{proof}
Note first that $L\otimes\sigma^*L\cong \pi^*(K_Y\otimes\Delta)\cong K_X$.	Let $M\in{\rm Pic}^0(X)$ such that $h^0(L\otimes M)=0$. We have \begin{align*}
	    h^0(L\otimes\sigma^*M^{-1}) &= h^1(L\otimes \sigma^*M^{-1})\\ &= h^0(L^{-1}\otimes \sigma^*M\otimes K_X)\\ &= h^0(\sigma^*L\otimes \sigma^*M)\\&= h^0(L\otimes M)\\&=0 
	\end{align*} 
 Hence $h^0(L\otimes (M\oplus \sigma^*M^{-1}))=0$. Let  $E=(M\oplus\sigma^*M^{-1})^{\oplus r}$. It is clearly a semistable anti-invariant vector bundle of rank $2r$ with a  $\sigma-$alternating isomorphism $\sigma^*E\ra E^*$  $$\psi=\begin{pmatrix}
	0 & 1 \\ -1 & 0
	\end{pmatrix}^{\oplus r}.$$  Now we have $h^0(E\otimes L)=0$, so the restriction of the divisor $\Theta_L$ to $\al U_X^{\sigma,-}(2r)$ is again a divisor.
	
	Moreover, by \cite[\S 7.10]{LS}, since the restriction of $\Theta_L$ to the moduli space $\al{U}_{X}^{\sigma,-}(2r)$ is a divisor, there exists an effective divisor ${\Xi}_L$ such that $2{\Xi}_L=\Theta_L|_{\al U_{X}^{\sigma,-}(2r)}$.  In particular,  $\sr P_L$ has a non-zero global section.  
	%Let  $\kappa_0$ be a theta characteristic on $Y$ such that the restriction of $\theta_\kappa=T_\kappa^*\theta$ to the (connected) Prym variety $\al P$ is a divisor, where $\theta\subset \x{Pic}^{g_X-1}(X)$  is the Riemann theta divisor and $T_\kappa:\x{Pic}^0(X)\ra \x{Pic}^{g_X-1}(X)$ is the translation map by $\kappa=\pi^*\kappa_0$. Then we see easily that such $\kappa$ gives the result. Indeed, if $L\in \al P$ such that $L\not\in \theta_\kappa$, then $L^{\oplus r}\in \al{U}_X^{\sigma,-}(r)\sm\Theta_\kappa$.   
\end{proof}
\begin{rema}
    Note that the restriction $\Xi_L$ to $\al{SU}_{X,\tau}^{\sigma,-}(2r)$ is again a divisor, for any type $\tau$. Indeed, if we set $\lambda=M\otimes \sigma^*M^{-1}$, then $E_\lambda=E\otimes\lambda^{-1}$ has trivial determinant (and it can have any type since it is strictly polystable) and since ${\rm Nm}_{X/Y}(\lambda)=\al O_Y$,  we have  $L\otimes \lambda\in{\rm Nm}_{X/Y}^{-1}(K_Y\otimes \Delta)$, and clearly  $h^0(L\otimes\lambda\otimes E_\lambda)=0$, so $E_\lambda\not\in\Xi_{L\otimes\lambda}$.\\
\end{rema}
    %{\color{red} How about $\al U_{X,1}^{\sigma,-}(2r)$?}
\section{Hitchin Systems on $\al U_X^{\sigma,-}(2r)$}

The Hitchin system is a Hamiltonian integrable system over the moduli space of vector bundles, introduced by N. Hitchin (\cite{NH}). It is a powerful tool for studying this moduli space, as well as the spaces of generalized theta functions. Let $\al U_X(r,0)$ be the moduli space of semistable vector bundles on $X$ of rank $r$ and degree $0$. The cotangent space to $\al U_X(r,0)$ at $E$ can be identified with $H^0(X,{\rm End}(E)\otimes K_X)$. By considering a basis of invariant polynomials under the adjoint action of $\ak{gl}_n$, we get a map $$T^*_E\al{U}_X(r,0)\ra \bigoplus_{i=0}^rH^0(X,K^i_X)=:W.$$ 
This is a Lagrangian fibration which is  known to be an algebraically completely integrable system, by the mean, its generic fiber is an open set of an abelian variety, which is, in this case, the Jacobian of the base curve.  

On the other hand, the cotangent space to $\al{U}_X^{\sigma,-}(2r)$ at $E$ can be identified with the eigenspace $H^0(X,{\rm End}(E)\otimes K_X)_-$ associated to the $-1$ eigenvalue of the involution induced by $\psi$ and the linearization on $K_X$ (see \cite{Z}). We can show that the above map induces a map    $$T^*_E\al{U}_X^{\sigma,-}(2r)\ra W^{\sigma,-}\subset \bigoplus_{i=0}^rH^0(X,K^i_X)_+,$$ where $H^0(X,K_X^i)_+$ is the invariant subspace and $W^{\sigma,-}$ is an affine subset defined by the condition that the characteristic polynomial is perfect square over each ramification point. \\
We gether some results in the following
\begin{prop}[\cite{Z}]\noindent
\begin{itemize}
    \item [$(i)$] $\dim W^{\sigma,-}=\dim \al U_X^{\sigma,-}(2r)=2r^2(g_X-1)-nr$.
    \item [$(ii)$] For general spectral data $s\in W^{\sigma,-}$, the associated spectral curve $\tilde X_s$ is singular with nodes as singularities. %at the ramification points of the degree $r$ cover  $q:\tilde X_s\ra X$.
    \item [$(iii)$] There exists an involution $\tilde \sigma$ on $\tilde X_s$  that lifts the involution $\sigma$ and whose fixed points set is precisely $\tilde R:=q^{-1}(R)$. Moreover $\tilde{R}$ is precisely the singular locus.
\end{itemize}
\end{prop}

Recall from loc.cit.  that, for general $s\in W^{\sigma,-}$,  we have the following diagram $$\xymatrix{\hat{X}_s\ar[r]^{\ak q}\ar[rd]^{\hat{\pi}} \ar@/^2pc/[rrr]^{\hat{q}}&  \tilde{X}_s \ar[rr]^q \ar[d]^{\tilde{\pi}} && X\ar[d]^\pi \\ &\tilde{Y}_s \ar[rr]^{\tilde{q}}&& Y,} $$
where $\hat X_s$ is the normalization of $\tilde{X}_s$ and  $\tilde{Y}_s=\tilde{X}_s/\tilde \sigma$. 

Denoted by $\hat{S}=\x{Ram}(\hat{X}_s/X)$  and by $\hat{\rm{P}}$ the variety of line bundles $L$ on $\hat{X}_s$ such that $$\hat{\sigma}^*L\cong L^{-1}(\hat{S}).$$ Note that ${\rm \hat{P}}={\rm{Nm}}_{\hat{X}_s/\tilde Y_s}^{-1}(M)$, where $M:=\al O_{\tilde{Y}_s}(S)\otimes\hat{\Delta}\otimes \tilde{q}^*\Delta^{-1}$, $S={\rm Ram}(\tilde{X}_s/\tilde Y_s)$ and $\hat{\Delta}=(\hat{\pi}_*\al O_{\hat{X}_s})^{-1}$ and $\Delta=({\pi}_*\al O_{{X}})^{-1}$ (see \cite{Z} middle of  page 32).  \\
Note that we have \begin{align*} \deg M& =2(g_{\hat{Y}_s}-1)-2r(g_Y-1)-rn \\&= g_{\hat{X}_s}-1-r(g_X-1).\end{align*}

\begin{prop}	For general $s\in W^{\sigma,-}$, we have 
\begin{enumerate}
    \item The pullback map $\hat{q}^*:{\rm Pic}^0(X)\lra{\rm Pic}^0(\hat{X}_s)$ is injective.% of degree $r^{2(n-1)}$.
 \item     The curve $\tilde Y_s=\tilde X_s/\tilde\sigma$ is smooth.
 \item The double cover $\hat X_s\ra \tilde Y_s$ is unramified.
\end{enumerate}
\end{prop}
\begin{proof}
\begin{enumerate}
    \item Note that  for general $s\in W^{\sigma,-}$, the cover $\tilde{X}_s\ra X$ does not factorize through an \'etale cover. So the same is true for $\hat{X}_s\ra X$. Now apply \cite[Proposition $11.4.3$]{BL}.
    \item This can be seen locally using the equation of $\tilde X_s$.
    \item Since the nodes are exactly $\rm Fix(\tilde\sigma)$, and since the action of $\tilde\sigma$ interchanges  the two sheets around any node, we deduce that $\hat\sigma$ has no fixed point. 
\end{enumerate}	
\end{proof}
In particular we deduce that the variety $\hat{\rm Q}=\x{Nm}_{\hat{X}_s/X}^{-1}(\hat{\delta})$ is connected, where $\hat{\delta}=\x{det}(\hat{q}_*\al O_{\hat{X}_s})^{-1}$. %${\rm\hat Q}=\x{Nm}_{\hat{X}_s/X}^{-1}(\hat{\delta})$, where   
 \\ Note also that we have  $$g_{\hat X_s}=4r^2(g_X-1)+1-2rn$$ and $$g_{\tilde Y_s}= 4r^2(g_Y-1)+nr(2r-1)+1.$$\\

\section{Main results}
Recall that we have constructed in \cite{Z} a dominant rational map $$\hat q_*:\hat{ \rm P}_0\dashrightarrow \al{U}_{X,0}^{\sigma,-}(2r),$$ where $\hat{\rm P}_0$ is a connected  component of $\hat{\rm P}$. Moreover, we have a dominant rational map $${\rm \hat P\cap \hat Q}\dashrightarrow \al{SU}_{X}^{\sigma,-}(2r).$$
Let $\rm{\hat S}$ be the connected component of $\rm\hat P\cap {\hat Q}$ that dominates $\al{SU}_{X,0}^{\sigma,-}(2r)$, where $0$ is the trivial type $(+1)_{p\in R}$.

%We fix an identification of $\hat{\rm P}_0$ with the Prym variety of $\hat{\pi}:\hat{X}_s\ra \tilde{Y}_s$ and 
We denote by  $\hat{\sr L}$ line bundles defining   principal polarization on $\rm{\hat{P}}_0$ such that $(\hat{q}_*)^*\sr P_L=\hat{\sr L}$. \\ 
We note that $\rm{\hat S}$ and $\rm P$ are complementary pair inside the principally polarized abelian variety $\rm \hat P_0$, that's we have an isogeny $${\rm\hat{S}\times P}\lra {\rm \hat{P}}_0,$$ given by $(L,M)\ra L\otimes q^*M$. Note that ${\rm\hat{S}\cap P}={\rm P}[r]$. The pullback of $\hat{\sr L}$ via this map is of the form $\hat{\sr M}\boxtimes \sr L$, where $\hat{\sr M}$ (resp. $\sr L$) is a polarization on $\rm \hat{S}$ (resp. $\rm P$).
\begin{prop} We have 
	$$H^0(\al{U}_{X,0}^{\sigma,-}(2r),\sr{P}_L)=1.$$
\end{prop}
\begin{proof}
% Let  Let $\kappa$ be a theta characteristic on $X$. 
Let $\rm T\subset \hat P_0$ be the locus of line bundles $L$ such that $q_*L$ is semistable.  By Lemma (\ref{lemm1}) below, and  %One can show, by using the same method as in  \cite{BNR}, that the codimension of the complement of $\rm S$ is at least $2$. 
 since $q_*$ is dominant, we get an injection given by pullback $$H^0(\al{U}_{X,0}^{\sigma,-}(2r),\sr P_L)\hookrightarrow H^0(\rm\hat P_0, \hat{\sr L}).$$ So the dimension of $ H^0(\al{U}_{X,0}^{\sigma,-}(2r),\sr P_L)$ is at most $1$.  The result now follows from Lemma \ref{effective}.% to get the result.
\end{proof}
\begin{lemm} \label{lemm1}
    The codimension of the complement of $\rm T$ is at least two.
\end{lemm}
\begin{proof}
    We emphasize that the proof is similar to \cite[Proposition $5.1$]{BNR}.\\ Let $\kappa$ be a theta characteristic on $X$. Consider  $\Theta_\kappa\subset \rm \hat{P}_0$ be the locus of line bundles $L$ such that $h^0(L\otimes \hat{q}^*\kappa)\not=0$. %Note that $\al O(\Theta_\kappa)$ is a ample line bundle. 
    We have $T^c\subset  \Theta_\kappa$. Indeed, let $L\not\in \Theta_\kappa$, assume that $q_*L$ is not semistable. So $q_*L$ has a subbundle $F$ such that $\deg F>0$, this implies $\deg F\otimes \kappa>{\rk (F)}(g-1)$, in particular, by Riemann-Roch, $F\otimes \kappa$ has a nonzero global section, this contradicts the assumption $L\not\in \Theta_\kappa.$ Now it is sufficient to show that  the inclusion  $T^c\subset  \Theta_\kappa$ is strict.   For this let $L\not\in\Theta_\kappa$, the positive dimensional variety $q^*J_X\otimes L:=\{q^*M\otimes L|M\in J_X\}$, which is contained in $T$, intersect the ample divisor $\Theta_\kappa$. Hence the result. 
\end{proof}
By \cite[Proposition $4.3$]{Z3}, the restriction of the Pfaffian line bundle $\sr P_L$ to $\al{SU}_{X,0}^{\sigma,-}(2r)$ is independent of $L$, so we denote it by $\sr P$. \\
We denote by $\al H(\sr L^r)$ the kernel of the polarization $$\Lambda(\sr L^r):\rm P\lra{\rm P}^*,$$  where $\rm P^*$ is the dual abelian variety. The associated theta group, which is a central extension of $\al H(\sr L^r)$, is denoted $\al G(\sr L^r)$.\\
We use similar notations for the line bundle $\hat{\sr L}$ on $\rm \hat{P}$.
\begin{prop}
	The type of the line bundle  $\hat{\sr M}$ on $\rm\hat S$ is given by $$(1,\dots,1,\underbrace{r,\dots,r}_{n-1},\underbrace{2r,\dots,2r}_{g_Y}).$$
\end{prop}
\begin{proof} The type of the polarization $\sr L$ on $\rm P$ is given by $$(1,\dots,1,\underbrace{2,\dots,2}_{g_Y}).$$ Hence the type of $\sr L^{r}$ is $(r,\dots,r,2r,\dots,2r).$ Note that $\rm{\hat S}$ and $\rm P$ are complementary pair inside the principally polarized abelian variety $\rm \hat P_0$ such that $\dim \rm \hat{S}>\dim\rm P$, so, we use \cite[Corollary $12.1.5$]{BL} to get the result.
\end{proof}
 \begin{prop}\label{Hgroup}
	The group $ {\rm P}[2r]$ of $2r-$torsion points of the Prym variety $\rm  P$ acts transitively in a natural way on the set of connected components $\pi_0(\rm{\hat P}_0\cap\rm{\hat Q})$. Moreover, the stabilizer of the identity component is given by $$ \x{Stab}({\rm \hat S})=\sqrt[r]{J_Y[2]}:=\{L\in{\rm P}[2r]\,|\, L^r\in \pi^*J_Y[2]\}.$$ In particular, the subgroup ${\rm P}[r]\subset{\rm \hat{P}}[2r]$ is included in $\x{Stab}(\hat{\rm S})$. Moreover,  we have  $$\al H(\sr{L}^r)=\x{Stab}(\hat{\rm{S}}). $$ 
\end{prop}
%We deduce that the theta group of the polarization $\sr{L}^r$ acts on $\al{SU}_{X,0}^{\sigma,-}(2r)$.
\begin{proof}
\cite[Proposition $4.27$ and Theorem $4.28$]{Z}.
\end{proof}
\begin{theo}\label{main}
 There exists a canonical isomorphism $$  H^0(\al{SU}_{X,0}^{\sigma,-}(2r),\sr P)\cong H^0({\rm P},\sr{L}^{r})^*.$$ In particular we deduce $$\dim(H^0(\al{SU}_{X,0}^{\sigma,-}(2r),\sr P))=2^{g_Y}r^{g_Y+n-1}.$$ 
\end{theo}
\begin{proof}
	%Consider the surjective  map $${ \al{SU}_X^{\sigma,+}(r)\times \al P\lra  \al{U}_{X,0}^{\sigma,+}(r), }$$ %\xymatrix{\sr{P}\cap \sr Q \times \al P\ar[rr]\ar[d] & & \sr P\ar[d]\\
	%where $\sr Q$ is the Prym variety of $\tilde{X}_s\ra X$. Using  \cite{BNR}, Theorem $3$, we deuce that  the pullback of the line bundle $\al O(\Xi_\kappa)$ via the above map is of the form  $p_1^*\al P\otimes p_2^*\al O(r\Xi)$, for some Pffafian bundle $\al P$ over $\al{SU}_X^{\sigma,+}(r)$. %Now  if $\sr S\subset \sr P\cap \sr Q$ is the locus of line bundles $L$ such that $q_*L$ is stable, then the codimension  of the complement of $\sr S$ is at least $2$, since $\sr S\lra \al{SU}_X^{\sigma,+}(r)$ is dominant, we deduce an injection $$H^0(\al{SU}_X^{\sigma,+}(r), \al P)\ra H^0(\sr P\cap \sr Q,\al L),$$ where $\al L$ is the restriction of the principal polarisation $\tilde{\Xi}$ of $\sr P$ to $\sr P\cap \sr Q$.
	%Moreover we have an isomorphism $$H^0(\sr P\cap \sr Q,\al L)\otimes H^0(\al P,r\Xi)\stackrel{\sim}{\lra} H^0(\sr P,\tilde{\Xi})\cong \bb C. $$ 
	%Moreover we have an isomorphism $$H^0(\al{SU}_X^{\sigma,+}(r),\al P)\otimes H^0(\al P,r\Xi)\stackrel{\sim}{\lra} H^0(\al{U}_{X,0}^{\sigma,+}(r),\al O(\Xi_\kappa))\cong \bb C. $$ This implies that $$H^0(P,r\Xi)^*\cong  H^0(\al{SU}_X^{\sigma,+}(r),\al P).$$
 Consider the following commutative diagram $$\xymatrix{\hat{\x{S}}\times \x P\ar[rr]\ar[d] & & \hat{\x P}_0\ar[d]\\ \al{SU}_{X,0}^{\sigma,-}(2r)\times {\rm P}\ar[rr] &&  \al U_{X,0}^{\sigma,-}(2r), }$$
		Using  \cite[Theorem $3$]{BNR}, we deduce that  the pullback of the line bundle $\sr P_L$ to $\al{SU}_{X,0}^{\sigma,-}(2r)\times \rm P$ is of the form  $p_1^*\sr P\otimes p_2^*\sr L^{r}$.\\ 
	Now the rational map $\hat{q}_*:\hat{\x{S}}\lra \al{SU}_{X,0}^{\sigma,-}(2r)$ is dominant by \cite[Theorem $4.16$]{Z}. The locus in $\rm \hat S$ of line bundles $L$ such that  $\hat{q}_*L$ is not stable has codimension at least two, it follows that the map $$(\hat q_*)^*:H^0(\al{SU}_{X,0}^{\sigma,-}(2r), \sr P)\rightarrow H^0(\hat{\x{S}}, \hat{\sr M})$$ is injective,
	%where we keep denoting by $\tilde{\sr L}$ the restriction of $\tilde{\sr L}$ to $\hat{\x{S}}\subset\hat{\rm P}_0$. \\
	%Now  if $\sr S\subset \sr P\cap \sr Q$ is the locus of line bundles $L$ such that $q_*L$ is stable, then the codimension  of the complement of $\sr S$ is at least $2$, since $\sr S\lra \al{SU}_X^{\sigma,+}(r)$ is dominant, we deduce an injection $$H^0(\al{SU}_X^{\sigma,+}(r), \al P)\ra H^0(\sr P\cap \sr Q,\al L),$$ where $\al L$ is the restriction of the principal polarization $\tilde{\Xi}$ of $\sr P$ to $\sr P\cap \sr Q$.
	%Moreover we have an isomorphism $$H^0(\sr P\cap \sr Q,\al L)\otimes H^0(\al P,r\Xi)\stackrel{\sim}{\lra} H^0(\sr P,\tilde{\Xi})\cong \bb C. $$ 
	%Moreover we have an isomorphism $$H^0(\al{SU}_X^{\sigma,+}(r),\al P)\otimes H^0(\al P,r\Xi)\stackrel{\sim}{\lra} H^0(\al U_X^{\sigma,+,0}(r),\al O(\tilde{\Theta}_\kappa))\cong \bb C. $$ 
	Moreover the group $\al H(\sr L^r)=\sqrt[r]{J_Y[2]}$ acts naturally on  $\al{SU}_{X,0}^{\sigma,-}(2r)$, and the theta group  $\al G(\sr P)$ of automorphisms of $\sr P$  acts  on  the  space $ H^0(\al{SU}_{X,0}^{\sigma,-}(2r),\sr P)$. Since $(\hat{q}_*)^*\sr P=\hat{\sr M}$, we deduce that $\al G(\sr P)\cong\al G(\hat{\sr M})$. This is because  we have $$1\ra \bb C^*\ra \al G(\hat{\sr M})\ra \al H(\hat{\sr M})\ra 1, $$ and  $$\al H(\hat{\sr M})\cong \al H(\sr L^r).$$ Since the  space $H^0(\hat{\rm S},\hat{\sr M})$ is an irreducible $\al G(\hat{\sr M})-$representation (\cite{M3}), the  map $$H^0(\al{SU}_{X,0}^{\sigma,-}(2r),\sr P)\hookrightarrow H^0(\hat{\rm{S}},\hat{\sr{M}}),$$ which is equivariant for these actions, is necessarily an isomorphism.\\
	Now, the two abelian subvarieties $\rm P$ and $\hat{\rm S}$ are  complementary pair inside $\hat{\rm{P}}_0$, we obtain  using \cite[Proposition $2.4$]{BNR}  an isomorphism $$H^0(\hat{\rm{S}}, \hat{\sr M})\cong H^0({\rm P},\sr{L}^{r})^*.$$
	Hence we deduce an isomorphism  $$H^0(\al{SU}_{X,0}^{\sigma,-}(2r),\sr P)\cong H^0({\rm P},\sr L^{r})^*.$$
\end{proof}

\bibliographystyle{acm}
\bibliography{bib}
\end{document}